\documentclass[12pt,oneside,reqno]{amsart}

\usepackage{amsmath,amssymb,amsthm,textcomp}
\usepackage{amsfonts,graphicx}
\usepackage[mathscr]{eucal}
\usepackage{color}
\usepackage{csquotes}
\usepackage[backend=bibtex,%
firstinits=true,%
doi=false,%
isbn=true,%
url=false,%
maxnames=99]{biblatex}%

\AtEveryBibitem{\clearfield{issn}}
\AtEveryCitekey{\clearfield{issn}}
\numberwithin{equation}{section}
\DeclareNameAlias{sortname}{last-first}
\theoremstyle{definition}
\usepackage{mathtools}
\numberwithin{equation}{section}

\newcommand{\ncom}{\newcommand}

\ncom{\beq}{\begin{equation}}
\ncom{\eeq}{\end{equation}}
\ncom{\bea}{\begin{eqnarray*}}
\ncom{\eea}{\end{eqnarray*}}
\ncom{\beqa}{\begin{eqnarray}}
\ncom{\eeqa}{\end{eqnarray}}
\ncom{\nno}{\nonumber}
\ncom{\non}{\nonumber}
\ncom{\ds}{\displaystyle}
\ncom{\half}{\frac{1}{2}}
\ncom{\mbx}{\makebox{.25cm}}
\ncom{\hs}{\mbox{\hspace{.25cm}}}
\ncom{\rar}{\rightarrow}
\ncom{\Rar}{\Rightarrow}
\ncom{\noin}{\noindent}
\ncom{\bc}{\begin{center}}
\ncom{\ec}{\end{center}}
\ncom{\sz}{\scriptsize}
\ncom{\rf}{\ref}
\ncom{\s}{\sqrt{2}}
\ncom{\sgm}{\sigma}
\ncom{\Sgm}{\Sigma}
\ncom{\psgm}{\sigma^{\prime}}
\ncom{\dt}{\delta}
\ncom{\Dt}{\Delta}
\ncom{\lmd}{\lambda}
\ncom{\Lmd}{\Lambda}
\ncom{\Th}{\Theta}
\ncom{\e}{\eta}
\ncom{\eps}{\epsilon}
\ncom{\pcc}{\stackrel{P}{>}}
\ncom{\lp}{\stackrel{L_{p}}{>}}
\ncom{\dist}{{\rm\,dist}}
\ncom{\sspan}{{\rm\,span}}
\ncom{\re}{{\rm Re\,}}
\ncom{\im}{{\rm Im\,}}
\ncom{\sgn}{{\rm sgn\,}}
\ncom{\ba}{\begin{array}}
\ncom{\ea}{\end{array}}
\ncom{\hone}{\mbox{\hspace{1em}}}
\ncom{\htwo}{\mbox{\hspace{2em}}}
\ncom{\hthree}{\mbox{\hspace{3em}}}
\ncom{\hfour}{\mbox{\hspace{4em}}}
\ncom{\vone}{\vskip 2ex}
\ncom{\vtwo}{\vskip 4ex}
\ncom{\vonee}{\vskip 1.5ex}
\ncom{\vthree}{\vskip 6ex}
\ncom{\vfour}{\vspace*{8ex}}
\ncom{\norm}{\|\;\;\|}
\ncom{\integ}[4]{\int_{#1}^{#2}\,{#3}\,d{#4}}
\ncom{\vspan}[1]{{{\rm\,span}\{ #1 \}}}
\ncom{\dm}[1]{ {\displaystyle{#1} } }
\ncom{\ri}[1]{{#1} \index{#1}}

\newtheorem{theorem}{\bf Theorem}[section]
\newtheorem{remark}{\bf Remark}[section]
\newtheorem{proposition}{Proposition}[section]

\newtheorem{definition}{Definition}[section]
\newtheoremstyle
    {remarkstyle}
    {}
    {11pt}
    {}
    {}
    {\bfseries}
    {:}
    {     }
    {\thmname{#1} \thmnumber{#2} }

\theoremstyle{remarkstyle}

\def\eps{\varepsilon}

\begin{document}
\title{Some Compound Fractional Poisson Processes}
\author[Mostafizar Khandakar]{Mostafizar Khandakar}
\address{Mostafizar Khandakar, Department of Mathematics,
	 Indian Institute of Technology Bombay, Powai, Mumbai 400076, India.}
\email{mostafizar@math.iitb.ac.in}
\author[Kuldeep Kumar Kataria]{Kuldeep Kumar Kataria}
\address{Kuldeep Kumar Kataria, Department of Mathematics, Indian Institute of Technology Bhilai, Raipur 492015, India.}
\email{kuldeepk@iitbhilai.ac.in}
\date{June 9, 2022}
\begin{abstract}
In this paper, we introduce and study fractional versions of three compound Poisson processes, namely, the Bell-Touchard process, the Poisson-logarithmic process and the generalized P\'olya-Aeppli process. It is shown that these processes are limiting cases of a recently introduced process by Di Crescenzo {\it et al.} (2016), namely, the generalized counting process. We obtain the mean, variance, covariance, long-range dependence property {\it etc.} for these processes. Also, we obtain several equivalent forms of the one-dimensional distribution of fractional versions of these processes.
\end{abstract}

\maketitle
\section{Introduction}
Di Crescenzo {\it et al.} (2016) introduced and studied a fractional counting process that performs $k$ kinds of jumps of amplitude $1,2,\dots,k$ with positive rates $\lambda_{1}, \lambda_{2},\dots,\lambda_{k}$, respectively. We call it the generalized fractional counting process (GFCP) and denote it by $\{M_{\beta}(t)\}_{t\ge0}$, $0<\beta\le 1$. It is defined as a counting process whose state probabilities $p_{\beta}(n,t)=\mathrm{Pr}\{M_{\beta}(t)=n\}$ satisfy the following system of fractional differential equations:
\begin{equation}\label{cre}
\frac{\mathrm{d}^{\beta}}{\mathrm{d}t^{\beta}}p_{\beta}(n,t)=-(\lambda_{1}+\lambda_{2}+\dots+\lambda_{k})p_{\beta}(n,t)+	\sum_{j=1}^{\min\{n,k\}}\lambda_{j}p_{\beta}(n-j,t),\ n\ge0,
\end{equation}
with initial condition
\begin{equation*}
p_{\beta}(n,0)=\begin{cases}
1,\ n=0,\\
0,\ n\ge1.
\end{cases}
\end{equation*}
Here, $\dfrac{\mathrm{d}^{\beta}}{\mathrm{d}t^{\beta}}$ is the Caputo fractional derivative defined as (see Kilbas {\it et al.} (2006))
\begin{equation}\label{caputo}
\frac{\mathrm{d}^{\beta}}{\mathrm{d}t^{\beta}}f(t):=\left\{
\begin{array}{ll}
\dfrac{1}{\Gamma{(1-\beta)}}\displaystyle\int^t_{0} (t-s)^{-\beta}f'(s)\,\mathrm{d}s,\ 0<\beta<1,\\\\
f'(t),\ \beta=1,
\end{array}
\right.
\end{equation}
and its Laplace transform is given by (see Kilbas {\it et al.} (2006), Eq. (5.3.3))
\begin{equation}\label{lc}
\mathcal{L}\left(\frac{\mathrm{d}^{\beta}}{\mathrm{d}t^{\beta}}f(t);s\right)=s^{\beta}\tilde{f}(s)-s^{\beta-1}f(0),\ s>0.
\end{equation}

For $\beta=1$, the GFCP reduces to the generalized counting process (GCP) $\{M(t)\}_{t\ge0}$ (see Di Crescenzo {\it et al.} (2016)). For $k=1$, the GFCP and the GCP reduces to the time fractional Poisson process (TFPP) (see Beghin and Orsingher (2009)) and the Poisson process, respectively. 
 
Let $\{Y_{\beta}(t)\}_{t\ge0}$ be the inverse stable subordinator, that is, the first passage time of a stable subordinator $\{D_{\beta}(t)\}_{t\ge0}$. A stable subordinator $\{D_{\beta}(t)\}_{t\ge0}$ is a one-dimensional L\'evy process with non-decreasing sample paths. It is known that (see Di Crescenzo {\it et al.} (2016))
\begin{equation*}
M_{\beta}(t)\stackrel{d}{=}M\left(Y_{\beta}(t)\right),\ t\ge0,
\end{equation*}
where $\stackrel{d}{=}$ denotes equality in distribution and $\{M(t)\}_{t\ge0}$ is independent of $\{Y_{\beta}(t)\}_{t\ge0}$. 

Kataria and Khandakar (2022) showed that several recently introduced counting processes such as the Poisson process of order $k$, P\'olya-Aeppli process of order $k$, P\'olya-Aeppli process, negative binomial process {\it etc.} are particular cases of the GCP. In Kataria and Khandakar (2021), the authors studied a limiting case of the GFCP, namely, the convoluted fractional Poisson process.
  
For $j\ge 1$, let $\{N_j(t)\}_{t\ge0}$ be a Poisson process with intensity $\lambda_{j}$. J\'anossy {\it et al.} (1950) considered a composed process $\left\{\sum_{j=1}^{\infty} jN_j(t)\right\}_{t\ge0}$ with $\sum_{j=1}^{\infty}\lambda_{j}<\infty$. The probability mass function (pmf) of $\{\xi(t)\}_{t\ge 0}$ where $\xi(t) = \sum_{j=1}^{\infty} jN_j(t)$ is given by (see J\'anossy {\it et al.} (1950), Eq. (2.18))
\begin{equation}\label{janina}
\mathrm{Pr}\{\xi(t)=n\} = \sum_{\Omega_{n}}\prod_{j=1}^{n}\frac{(t\lambda_j)^{x_j}}{x_j!}e^{-\sum_{j=1}^{\infty}t\lambda_j }.
\end{equation}
Here,
\begin{equation}\label{onn}
\Omega_n= \left\{(x_{1},x_{2},\dots,x_{n}):\sum_{j=1}^{n}jx_{j}=n,\ x_{j}\in \mathbb{N}\cup \{0\}\right\}.
\end{equation}
Its equivalent version is given by (see J\'anossy {\it et al.} (1950), Eq. (2.22))
\begin{equation}\label{manina}
\mathrm{Pr}\{\xi(t)=n\} = \sum_{\Omega^{n}_{k}}\prod_{j=1}^{k}\lambda_{x_j}\frac{t^{k}}{k!}e^{-\sum_{j=1}^{\infty}t\lambda_j },
\end{equation}
where
\begin{equation}\label{okn}
\Omega^{n}_{k}=\left\{(x_{1},x_{2},\dots,x_{k}):\sum_{i=1}^{k}x_{i}=n,\ k\le n,\ x_{i}\ge 1\right\}.
\end{equation}

Let $\{X_{i}\}_{i\ge 1}$ be a sequence of independent and identically distributed (iid) random variables such that $\mathrm{Pr}\{X_{1}=j\}=c_{j}$ for all $j\ge 1$. Consider a compound Poisson process  
\begin{equation*}
Y(t)=\sum_{i=1}^{N(t)}X_{i},\ t\ge0,
\end{equation*}
where $\{N(t)\}_{t\ge0}$ is a Poisson process with intensity $\lambda>0$ that is independent of $\{X_{i}\}_{i\ge 1}$. 

For suitable choice of $\lambda$ and $c_{j}$'s the compound Poisson process $\{Y(t)\}_{t\ge0}$ is equal in distribution to a counting process introduced and studied by Freud and Rodriguez (2022), namely, the Bell-Touchard process (BTP) $\{\mathcal{M}(t)\}_{t\ge0}$. For $\lambda=\alpha(e^{\theta}-1)$ and $c_{j}=\theta^{j}/j!(e^{\theta}-1)$, $j\ge 1$ where $\alpha>0$, $\theta>0$ they have shown that
\begin{equation}\label{ajjagh12}
\mathcal{M}(t)\stackrel{d}{=}Y(t).
\end{equation}

For $\lambda_{j}=\alpha \theta^{j}/j!$, $j\ge 1$, the process  $\{\xi(t)\}_{t\ge 0}$ is equal in distribution to BTP (see Freud and Rodriguez (2022)), that is, $\xi(t)\stackrel{d}{=}\mathcal{M}(t)$. So, the pmf $q(n,t)=\mathrm{Pr}\{\mathcal{M}(t)=n\}$ of BTP is given by
\begin{equation}\label{p(n,t)11}
	q(n,t)=\sum_{\Omega_{n}}\prod_{j=1}^{n}\frac{(\alpha t\theta^{j}/j!)^{x_{j}}}{x_{j}!}e^{-\alpha t\left(e^{\theta}-1\right)},
\end{equation}
where $\Omega_{n}$ is given in \eqref{onn}.
 
For $c_{j}=-(1-p)^{j}/j\ln p$, $j\ge 1$ where $p\in (0,1)$, the compound Poisson process $\{Y(t)\}_{t\ge0}$ reduces to a counting process introduced and studied by Sendova and Minkova (2018), namely, the Poisson-logarithmic process (PLP). We denote it by $\{\hat{\mathcal{M}}(t)\}_{t\ge0}$. It is defined as
\begin{equation}\label{plp1}
\hat{\mathcal{M}}(t)=Y(t).
\end{equation}
As an application, they considered a risk process in which the PLP is used to model the claim numbers.
 
For $c_{j}=\binom{r+j-1}{j}\rho^{j}(1-\rho)^{r}/(1-(1-\rho)^{r})$, $j\ge 1$ where $r>0$, $0<\rho<1$, the compound Poisson process $\{Y(t)\}_{t\ge0}$ reduces to a counting process introduced and studied by Jacob and Jose (2018), namely, the generalized P\'olya-Aeppli process (GPAP). We denote it by $\{\bar{\mathcal{M}}(t)\}_{t\ge0}$. It is defined as
\begin{equation}\label{gpap1}
\bar{\mathcal{M}}(t)=Y(t).
\end{equation}
If $r=1$ then the GPAP reduces to the P\'olya-Aeppli process (see Chukova and Minkova (2013)).

In this paper, we study in detail the fractional versions of BTP, PLP and GPAP. We obtain their L\'evy measures. It is shown that these processes are the limiting cases of GCP. We obtain the probability generating function (pgf), mean, variance, covariance {\it etc.} for their fractional variants and establish their long-range dependence (LRD) property. Several equivalent forms of the pmf of fractional versions of these processes are obtained. We have shown that the fractional variants of these processes are overdispersed and non-renewal. It is observed that the one-dimensional distributions of their fractional variants are not infinitely divisible and their increments has the short-range dependence (SRD) property. It is shown that these fractional processes are equal in distribution to some particular cases of the compound fractional Poisson process studied by Beghin and Macci (2014).

\section{Preliminaries}
In this section, we give some known results and definitions that will be used later.
\subsection{Mittag-Leffler function}
The three-parameter Mittag-Leffler function is defined as (see Kilbas {\it et al.} (2006), p. 45)
\begin{equation*}
E_{\alpha,\beta}^{\gamma}(x)\coloneqq\frac{1}{\Gamma(\gamma)}\sum_{j=0}^{\infty} \frac{\Gamma(j+\gamma)x^{j}}{j!\Gamma(j\alpha+\beta)},\ x\in\mathbb{R},
\end{equation*}
where $\alpha>0$, $\beta>0$ and $\gamma>0$. 

It reduces to the two-parameter Mittag-Leffler function for $\gamma=1$, and for $\gamma=\beta=1$ it reduces to the Mittag-Leffler function. 

For any real $x$, the Laplace transform of the function $t^{\beta-1}E^{\gamma}_{\alpha,\beta}\left(xt^{\alpha}\right)$ is given by  (see Kilbas {\it et al.} (2006), Eq. (1.9.13)):
\begin{equation}\label{mi}
\mathcal{L}\left(t^{\beta-1}E^{\gamma}_{\alpha,\beta}\left(xt^{\alpha}\right);s\right)=\frac{s^{\alpha\gamma-\beta}}{(s^{\alpha}-x)^{\gamma}},\ s>|x|^{1/\alpha}.
\end{equation}

\subsection{Inverse stable subordinator}
The following result holds for the inverse stable subordinator $\{Y_{\beta}(t)\}_{t\ge0}$ (see Leonenko {\it et al.} (2014)):
\begin{equation}\label{ybx}
\mathbb{E}\left(e^{-sY_{\beta}(t)}\right)=E_{\beta,1}\left(-st^{\beta}\right),\ s>0.
\end{equation}	
The mean and variance of $\{Y_{\beta}(t)\}_{t\ge0}$ are given by (see Leonenko {\it et al.} (2014))
\begin{align}\label{meani}
\mathbb{E}\left(Y_{\beta}(t)\right)&=\frac{t^{\beta}}{\Gamma(\beta+1)},\\
\operatorname{ Var}\left(Y_{\beta}(t)\right)&=\left(\frac{2}{\Gamma(2\beta+1)}-\frac{1}{\Gamma^{2}(\beta+1)}\right)t^{2\beta}.\label{xswe331}
\end{align}

For fixed $s$ and large $t$, the following asymptotic result holds (see Maheshwari and Vellaisamy (2016)):
\begin{equation}\label{asi1}
\operatorname{Cov}\left(Y_{\beta}(s),Y_{\beta}(t)\right)\sim \frac{1}{\Gamma^2(\beta+1)}\left( \beta s^{2\beta}B(\beta,\beta+1)-\frac{\beta^{2}s^{\beta+1}}{(\beta+1)t^{1-\beta}}\right),
\end{equation}  
where $B(\beta,\beta+1)$ denotes the beta function.
\subsection{Poisson-logarithmic process} 
The state probabilities $\hat{q}(n,t)=\mathrm{Pr}\{\hat{\mathcal{M}}(t)=n\}$ of PLP satisfies the following (see Sendova and Minkova (2018)):
\begin{equation}\label{btpgov3}
\begin{aligned}
\frac{\mathrm{d}}{\mathrm{d}t}\hat{q}(0,t)&=-\lambda \hat{q}(0,t),\\
\frac{\mathrm{d}}{\mathrm{d}t}\hat{q}(n,t)&=-\lambda \hat{q}(n,t)-\frac{\lambda}{\ln p}\sum_{j=1}^{n}\frac{(1-p)^{j}}{j}\hat{q}(n-j,t),\ n\ge 1,
\end{aligned}
\end{equation}
with initial conditions $\hat{q}(0,0)=1$ and $\hat{q}(n,0)=0$, $n\ge1$. Its pgf $\hat{G}(u,t)=\mathbb{E}\left(u^{\hat{\mathcal{M}}(t)}\right)$ is given by
\begin{equation}\label{pgfplp}
\hat{G}(u,t)=\exp \left(-\lambda t\left(1-\frac{\ln(1-(1-p)u)}{\ln p}\right)\right).
\end{equation}
 
\subsection{Generalized P\'olya-Aeppli process} 
The state probabilities $\bar{q}(n,t)=\mathrm{Pr}\{\bar{\mathcal{M}}(t)=n\}$ of GPAP satisfies the following (see Jacob and Jose (2018)):
\begin{equation}\label{btpgov33}
\begin{aligned}
\frac{\mathrm{d}}{\mathrm{d}t}\bar{q}(0,t)&=-\lambda \bar{q}(0,t),\\
\frac{\mathrm{d}}{\mathrm{d}t}\bar{q}(n,t)&=-\lambda \bar{q}(n,t)+\frac{\lambda}{(1-\rho)^{-r}-1}\sum_{j=1}^{n}\binom{r+j-1}{j}\rho^{j}\bar{q}(n-j,t),\ n\ge 1,
\end{aligned}
\end{equation}
with initial conditions $\bar{q}(0,0)=1$ and $\bar{q}(n,0)=0$, $n\ge1$. Its pgf $\bar{G}(u,t)=\mathbb{E}\left(u^{\bar{\mathcal{M}}(t)}\right)$ is given by 
\begin{equation}\label{pgfplp3}
\bar{G}(u,t)=\exp \left(-\lambda t\left(1-\frac{(1-\rho u)^{-r}-1}{(1-\rho)^{-r}-1}\right)\right).
\end{equation}
Let $\bar{r}_{1}=r\rho \lambda /(1-\rho)(1-(1-\rho)^{r})$ and $\bar{r}_{2}=r(1+r\rho)\rho \lambda/(1-\rho)^{2}(1-(1-\rho)^{r})$. Its mean and variance are given by
\begin{equation}\label{meanvar}
\mathbb{E}\left(\bar{\mathcal{M}}(t)\right)=\bar{r}_{1}t, \ \operatorname{Var}\left(\bar{\mathcal{M}}(t)\right)=\bar{r}_{2}t.
\end{equation}

\subsection{LRD and SRD properties}
The LRD and SRD properties for a non-stationary stochastic process $\{X(t)\}_{t\geq0}$ are defined as follows (see D'Ovidio and Nane (2014), Maheshwari and Vellaisamy (2016)):
\begin{definition}
Let $s>0$ be fixed. The process $\{X(t)\}_{t\ge0}$ is said to exhibit the LRD property if its correlation function has the following asymptotic behaviour:
\begin{equation*}
\operatorname{Corr}(X(s),X(t))\sim c(s)t^{-\theta},\  \text{as}\ t\rightarrow\infty,
\end{equation*}
for some $c(s)>0$ and $\theta\in(0,1)$. If $\theta\in(1,2)$ then it is said to possesses the SRD property.
\end{definition}
\section{Bell-Touchard process and its fractional version}
Here, we introduce a fractional version of a recently introduced counting process, namely, the Bell-Touchard process (BTP) (see Freud and Rodriguez (2022)).
 
First, we study some additional properties of BTP which is defined as follows:
\begin{definition}
A counting process $\{\mathcal{M}(t)\}_{t\ge0}$ is said to be a BTP with parameters $\alpha>0$, $\theta>0$ if\\
(a) $\mathcal{M}(0)=0$;\\	
(b) it has independent and stationary increments;\\	
(c) for $m=0,1,\dots$ and $h>0$ small enough such that $o(h)\to 0$ as $h\to 0$, its state transition probabilities are given by
\begin{equation*}
\mathrm{Pr}\{\mathcal{M}(t+h)=n|\mathcal{M}(t)=m\}=\begin{cases*}	1-\alpha \left(e^{\theta}-1\right)h+o(h),\ n=m,\\
\alpha\theta^{j}h/j!+o(h),\ n=m+j,\ j=1,2,\dots.
\end{cases*}
\end{equation*}
\end{definition}
It follows that the state probabilities $q(n,t)=\mathrm{Pr}\{\mathcal{M}(t)=n\}$, $n\ge0$ of BTP satisfy the following system of differential equations:
\begin{equation}\label{btpgov}
\begin{aligned}
\frac{\mathrm{d}}{\mathrm{d}t}q(0,t)&=-\alpha \left(e^{\theta}-1\right)q(0,t),\\
\frac{\mathrm{d}}{\mathrm{d}t}q(n,t)&=-\alpha \left(e^{\theta}-1\right)q(n,t)+\alpha\sum_{j=1}^{n}\frac{\theta^{j}}{j!}q(n-j,t),\ n\ge 1,
\end{aligned}
\end{equation}
with initial conditions $q(0,0)=1$ and $q(n,0)=0$, $n\ge1$.
\begin{remark}\label{rem1}
On taking $\beta=1$, $\lambda_{j}=\alpha \theta^{j}/j!$ for all $j\ge 1$ and letting $k\to \infty$, the System (\ref{cre}) reduces to the System (\ref{btpgov}). Thus, the BTP is a limiting case of the GCP.
\end{remark}
In view of Remark \ref{rem1}, we note that several results for BTP can be obtained from the corresponding results for GCP. Next, we present a few of them.

The next result gives a recurrence relation for the pmf of BTP. It follows from Proposition 1 of Kataria and Khandakar (2022).
\begin{proposition}
The  state probabilities $q(n,t)$ of BTP satisfy
\begin{equation*}
q(n,t)=\frac{\alpha t}{n}\sum_{j=1}^{n}\frac{\theta^{j}}{(j-1)!}q(n-j,t),\ n\ge1.	
\end{equation*}
\end{proposition}

It is known that the GCP is equal in distribution to the following weighted sum of $k$ independent Poisson processes (see Kataria and Khandakar (2022)):
\begin{equation*}
M(t)\stackrel{d}{=}\sum_{j=1}^{k}jN_{j}(t),
\end{equation*}
where $\{N_{j}(t)\}_{t\ge0}$'s are independent Poisson processes with intensity $\lambda_{j}$'s. On taking $\lambda_{j}=\alpha \theta^{j}/j!$ for all $j\ge 1$ and letting $k\to \infty$, we get
\begin{equation*}
\mathcal{M}(t)\stackrel{d}{=}\sum_{j=1}^{\infty}jN_{j}(t),
\end{equation*}	
which agrees with the result obtained by Freud and Rodriguez (2022). As $\underset{t\to\infty}{\lim}N_{j}(t)/t=\alpha \theta^{j}/j!$, $j\ge1$ a.s., we have 
\begin{align}
\lim_{t\to\infty}\frac{\mathcal{M}(t)}{t}&\stackrel{d}{=}\sum_{j=1}^{\infty}j\lim_{t\to\infty}\frac{N_{j}(t)}{t}\nonumber\\
&=\alpha \theta e^{\theta},\ \text{in probability}.\label{lemmaar}
\end{align}

The next result gives a martingale characterization for the BTP whose proof follows from Proposition 2 of Kataria and Khandakar (2022).
\begin{proposition}
The process $\left\{\mathcal{M}(t)-\alpha \theta e^{\theta}t\right\}_{t\geq0}$ is a martingale with respect to a natural filtration $\mathscr{F}_{t}=\sigma\left(\mathcal{M}(s), s\le t\right)$.
\end{proposition}

From (\ref{ajjagh12}), it follows that the BTP is a L\'evy process as it is equal in distribution to a compound Poisson process. So, its mean, variance and covariance are given by
\begin{align}
E(\mathcal{M}(t))&=\alpha \theta e^{\theta}t,\label{meanbtp}\\
\operatorname{Var}(\mathcal{M}(t))&=\alpha \theta(\theta+1) e^{\theta}t,\label{varbtp}\\
\operatorname{Cov}(\mathcal{M}(s),\mathcal{M}(t))&=\alpha \theta(\theta+1) e^{\theta}\min\{s,t\},\nonumber
\end{align}
respectively. The BTP exhibits the overdispersion property  as $\operatorname{Var}(\mathcal{M}(t))-E(\mathcal{M}(t))=\alpha \theta^{2} e^{\theta}t>0$ for $t>0$. 

The characteristic function of BTP can be obtained by taking $\lambda_{j}=\alpha \theta^{j}/j!$ for all $j\ge 1$ and letting $k\to \infty$ in Eq. (12) of Kataria and Khandakar (2022). It is given by 
\begin{equation*}
\mathbb{E}\left(e^{\omega\xi \mathcal{M}(t)}\right)=\exp\left(-\alpha t\left(e^{\theta}-e^{\theta e^{\omega\xi}}\right)\right),\ \omega=\sqrt{-1},\ \xi\in\mathbb{R}.
\end{equation*}
So, its L\'evy measure is given by
$
\mu(\mathrm{d}x)=\displaystyle\alpha\sum_{j=1}^{\infty}\frac{\theta^{j}}{j!}\delta_{j}\mathrm{d}x,
$
where $\delta_{j}$'s are Dirac measures.

\begin{remark}
For fixed $s$ and large $t$, the correlation function of BTP has the following
asymptotic behaviour: 
\begin{equation*}
\operatorname{Corr}(\mathcal{M}(s),\mathcal{M}(t))\sim \sqrt{s}t^{-1/2}.
\end{equation*}
Thus, it exhibits the LRD property.
\end{remark}
\subsection{Fractional Bell-Touchard process}
Here, we introduce a fractional version of the BTP, namely, the fractional Bell-Touchard process (FBTP). We define it as the stochastic process $\{\mathcal{M}_{\beta}(t)\}_{t\ge0}$, $0<\beta\le 1$ whose state probabilities  $q_{\beta}(n,t)=\mathrm{Pr}\{\mathcal{M}_{\beta}(t)=n\}$ satisfy the following system of fractional differential equations:
\begin{equation}\label{fbtpgov}
\begin{aligned}
\frac{\mathrm{d}^{\beta}}{\mathrm{d}t^{\beta}}q_{\beta}(0,t)&=-\alpha \left(e^{\theta}-1\right)q_{\beta}(0,t),\\
\frac{\mathrm{d}^{\beta}}{\mathrm{d}t^{\beta}}q_{\beta}(n,t)&=-\alpha \left(e^{\theta}-1\right)q_{\beta}(n,t)+\alpha\sum_{j=1}^{n}\frac{\theta^{j}}{j!}q_{\beta}(n-j,t),\ n\ge 1,
\end{aligned}
\end{equation}
with initial conditions $q_{\beta}(0,0)=1$ and $q_{\beta}(n,0)=0$, $n\ge1$.

Note that the system of equations (\ref{fbtpgov}) is obtained by replacing the integer order derivative in (\ref{btpgov}) by Caputo fractional derivative defined in \eqref{caputo}.
\begin{remark}
On taking $\lambda_{j}=\alpha \theta^{j}/j!$ for all $j\ge 1$ and letting $k\to \infty$, the System (\ref{cre}) reduces to the System (\ref{fbtpgov}). Thus, the FBTP is a limiting case of the GFCP.
\end{remark}

Using \eqref{fbtpgov}, it can be shown that the pgf $G_{\beta}(u,t)=\mathbb{E}\left(u^{\mathcal{M}_{\beta}(t)}\right)$, $|u|\leq1$ of FBTP satisfies
\begin{equation*}
\frac{\mathrm{d}^{\beta}}{\mathrm{d}t^{\beta}}G_{\beta}(u,t)=-\alpha \left(e^{\theta}-e^{\theta u}\right)G_{\beta}(u,t),
\end{equation*}
with $G_{\beta}(u,0)=1$. On taking the Laplace transform in the above equation and using \eqref{lc}, we get
\begin{equation*}
s^{\beta}\tilde{G}_{\beta}(u,s)-s^{\beta-1}G_{\beta}(u,0)=-\alpha \left(e^{\theta}-e^{\theta u}\right)\tilde{G}_{\beta}(u,s),\ s>0,
\end{equation*}
where $\tilde{G}_{\beta}(u,s)$ denotes the Laplace transform of $G_{\beta}(u,t)$. Thus,
\begin{equation*}
\tilde{G}_{\beta}(u,s)=\frac{s^{\beta-1}}{s^{\beta}+\alpha \left(e^{\theta}-e^{\theta u}\right)}.
\end{equation*}
On taking inverse Laplace transform and using (\ref{mi}), we get
\begin{equation}\label{pgfftbp}
G_{\beta}(u,t)=	E_{\beta,1}\left(-\alpha \left(e^{\theta}-e^{\theta u}\right)t^{\beta}\right).
\end{equation}
\begin{remark}\label{rem23}
On taking $\beta=1$ in \eqref{pgfftbp}, we get the pgf $G(u,t)=\exp\left(-\alpha \left(e^{\theta}-e^{\theta u}\right)t\right)$ of BTP. Also, from \eqref{pgfftbp} we can verify that the pmf $q_{\beta}(n,t)$ sums up to one, that is, $\sum_{n=0}^{\infty}q_{\beta}(n,t)=G_{\beta}(u,t)|_{u=1}=E_{\beta,1}(0)=1$.
\end{remark}
The next result gives a time-changed relationship between the BTP and its fractional variant, FBTP.
\begin{theorem}\label{thm4.1}
Let $\{Y_{\beta}(t)\}_{t\ge0}$, $0<\beta<1$, be an inverse stable subordinator independent of the BTP $\{\mathcal{M}(t)\}_{t\ge0}$. Then
\begin{equation}\label{ybt}
\mathcal{M}_{\beta}(t)\stackrel{d}{=}\mathcal{M}(Y_{\beta}(t)),\ t\ge0.
\end{equation}
\end{theorem}
\begin{proof}
Let $h_{\beta}(x,t)$ be the density of $\{Y_{\beta}(t)\}_{t\ge0}$.	Then, 
\begin{align*}
\mathbb{E}\left(u^{\mathcal{M}(Y_{\beta}(t))}\right)&=\int_{0}^{\infty}G(u,x)h_{\beta}(x,t)\mathrm{d}x\\
&=\int_{0}^{\infty}	\exp\left(-\alpha \left(e^{\theta}-e^{\theta u}\right)x\right)h_{\beta}(x,t)\mathrm{d}x,\ \ (\text{using}\ \text{Remark}\ \ref{rem23})\\
&=E_{\beta,1}\left(-\alpha \left(e^{\theta}-e^{\theta u}\right)t^{\beta}\right),\ (\text{using} \ \eqref{ybx}),
\end{align*}
which agrees with \eqref{pgfftbp}. This completes the proof.
\end{proof}
\begin{remark}\label{remark3.5}
Let $\{T_{2\beta}(t)\}_{t>0}$ be a random process whose distribution is given by the folded solution of the following Cauchy problem (see Orsingher and Beghin (2004)): 
\begin{equation}\label{diff}
\frac{\mathrm{d}^{2\beta}}{\mathrm{d}t^{2\beta}}u(x,t)=\frac{\partial^{2}}{\partial x^{2}}u(x,t),\ x\in\mathbb{R},\ t>0,
\end{equation}
with $u(x,0)=\delta(x)$ for $0<\beta\le 1$ and $\frac{\partial^{}}{\partial t}u(x,0)=0$ for $1/2<\beta\le1$. It is known that the density functions of $Y_{\beta}(t)$ and $T_{2\beta}(t)$ coincides (see Meerschaert {\it et al.} (2011)). Hence,
\begin{equation}\label{keyyekk1}
\mathcal{M}_{\beta}(t)\overset{d}{=}\mathcal{M}(T_{2\beta}(t)),\ t>0,
\end{equation}	
where $\{T_{2\beta}(t)\}_{t>0}$ is independent of the BTP.
	
The random process $\{T_{2\beta}(t)\}_{t>0}$ becomes a reflecting Brownian motion $\{|B(t)|\}_{t>0}$ for $\beta=1/2$ as the equation (\ref{diff}) reduces to the following heat equation:
\begin{equation*}
\begin{cases*}
\frac{\partial}{\partial t}u(x,t)=\frac{\partial^{2}}{\partial x^{2}}u(x,t),\ x\in\mathbb{R},\ t>0,\\
u(x,0)=\delta(x).
\end{cases*}
\end{equation*}
So, $\mathcal{M}_{1/2}(t)$ is equal in distribution to BTP at a Brownian time, that is, $\mathcal{M}(|B(t)|)$, $t>0$.
\end{remark}
In view of (\ref{ajjagh12})	and (\ref{ybt}), it follows that the FBTP is equal in distribution to the following compound fractional Poisson process:
\begin{equation}\label{compoundr1}
\mathcal{M}_{\beta}(t)\stackrel{d}{=}\sum_{i=1}^{N(Y_{\beta}(t))}X_{i}\stackrel{d}{=}\sum_{i=1}^{N_{\beta}(t)}X_{i},\ t\ge0,
\end{equation}
where $\{N_{\beta}(t)\}_{t\ge0}$ is a TFPP with intensity $\alpha\left(e^{\theta}-1\right)$ independent of the sequence of iid random variables $\{X_{i}\}_{i\ge1}$. Thus, it is neither Markovian nor a L\'evy process (see Scalas (2012)). 

\begin{remark}
The system of differential equations that governs the state probabilities of compound fractional Poisson process was obtained by Beghin and Macci (2014). In view of \eqref{compoundr1}, the System \eqref{fbtpgov} can alternatively be obtained using Proposition 1 of Beghin and Macci (2014).
\end{remark}	

Next, we obtain the pmf of FBTP and some of its equivalent versions. 

The solution $u_{2\beta}(x,t)$ of \eqref{diff} is given by (see Beghin and Orsingher (2009)):
\begin{equation}\label{u2b}
u_{2\beta}(x,t)=\frac{1}{2t^{\beta}}W_{-\beta, 1-\beta}\left(-\frac{|x|}{t^{\beta}}\right),\ t>0, \ x\in \mathbb{R},
\end{equation}
where $W_{\nu,\gamma}(\cdot)$ is the Wright function defined as follows:
\begin{equation*}
W_{\nu,\gamma}(x)=\sum_{k=0}^{\infty}\frac{x^{k}}{k!\Gamma(k\nu+\gamma)},\ \nu>-1,\ \gamma >0,\ x\in \mathbb{R}.
\end{equation*}
Let 
\begin{equation}\label{ub}
\bar{u}_{2\beta}(x,t)=\begin{cases*}
2u_{2\beta}(x,t),\ x>0, \\
0,\ x<0,
\end{cases*}
\end{equation}
be the folded solution to \eqref{diff}.
\begin{theorem}\label{thm4.21}
The pmf $q_{\beta}(n,t)=\mathrm{Pr}\{\mathcal{M}_{\beta}(t)=n\}$ of FBTP is given by
\begin{equation}\label{pmf1}
q_{\beta}(n,t)=\begin{cases*}	E_{\beta, 1}\left(-\alpha (e^{\theta}-1) t^{\beta}\right),\ n=0,\\
\displaystyle\sum_{\Omega_{n}}\prod_{j=1}^{n}\frac{\left(\alpha \theta^{j}/j!\right)^{x_{j}}}{x_{j}!}z_{n}!t^{z_{n}\beta }E_{\beta, z_{n}\beta +1}^{z_{n}+1}\left(-\alpha (e^{\theta}-1)t^{\beta}\right),\ n\ge1,
\end{cases*}
\end{equation}
where  $z_{n}=x_{1}+x_{2}+\dots+x_{n}$ and $\Omega_{n}$ is given in \eqref{onn}.
\end{theorem}
\begin{proof}
From \eqref{keyyekk1} and \eqref{ub}, we have
\begin{equation*}
q_{\beta}(n,t)=\int_{0}^{\infty}q(n,x)\bar{u}_{2\beta}(x,t)\mathrm{d}x.	
\end{equation*}
We use \eqref{p(n,t)11} and \eqref{u2b} in the above equation to obtain
\begin{align*}
q_{\beta}(n,t)&=\sum_{\Omega_{n}}\prod_{j=1}^{n}\frac{\left(\alpha \theta^{j}/j!\right)^{x_{j}}}{x_{j}!}t^{-\beta}\int_{0}^{\infty}e^{-\alpha x(e^{\theta}-1)}x^{z_{n}}W_{-\beta, 1-\beta}\left(-\frac{x}{t^{\beta}}\right)\mathrm{d}x \\
&=\sum_{\Omega_{n}}\prod_{j=1}^{n}\frac{\left(\alpha \theta^{j}/j!\right)^{x_{j}}}{x_{j}!}\frac{t^{-\beta}}{\left(\alpha (e^{\theta}-1)\right)^{z_{n}+1}}\int_{0}^{\infty}e^{-y}y^{z_{n}}W_{-\beta, 1-\beta}\left(-\frac{y}{\alpha (e^{\theta}-1)t^{\beta}}\right)\mathrm{d}y.
\end{align*}
On using the following result (see Beghin and Orsingher (2010), Eq. (2.13)):
\begin{equation*}
E_{\beta, z_{n}\beta +1}^{z_{n}+1}\left(-\alpha (e^{\theta}-1)t^{\beta}\right)=\frac{t^{-\beta(z_{n}+1)}}{z_{n}!\left(\alpha (e^{\theta}-1)\right)^{z_{n}+1}}\int_{0}^{\infty}e^{-y}y^{z_{n}}W_{-\beta, 1-\beta}\left(-\frac{y}{\alpha (e^{\theta}-1)t^{\beta}}\right)\mathrm{d}y,
\end{equation*}
the proof follows.
\end{proof}

\begin{remark}
An equivalent form of the pmf of BTP can be obtained from \eqref{manina}. It is given by 
\begin{equation}\label{alterpnt1}
q(n,t)=\sum_{\Omega^{n}_{k}}\prod_{j=1}^{k}\frac{\theta^{x_{j}}}{x_{j}!}\frac{(\alpha t)^{k}}{k!}e^{-\alpha t\left(e^{\theta}-1\right)},\ n\ge0,
\end{equation}
where $\Omega^{n}_{k}$ is given in \eqref{okn}.
If we use \eqref{alterpnt1} in the proof of Theorem \ref{thm4.21} then we can obtain the following alternate form of the pmf of FBTP: 
\begin{equation}\label{pmf2}
q_{\beta}(n,t)=\sum_{\Omega^{n}_{k}}\prod_{j=1}^{k}\frac{\theta^{x_{j}}}{x_{j}!}\left(\alpha t^{\beta}\right)^{k}E_{\beta, k\beta+1}^{k+1}\left(-\alpha (e^{\theta}-1)t^{\beta}\right).
\end{equation}
\end{remark}

The pmf of TFPP with intensity $\alpha\left(e^{\theta}-1\right)$ is given by (see Beghin and Orsingher (2010), Eq. (2.5))
\begin{equation*}
\mathrm{Pr}\{N_{\beta}(t)=n\}=\left(\alpha(e^{\theta}-1)t^{\beta}\right)^{n}E_{\beta,n\beta+1}^{n+1}\left(-\alpha(e^{\theta}-1)t^{\beta}\right),\ n\ge0.
\end{equation*}
From (\ref{compoundr1}), we have
\begin{equation*}
q_{\beta}(0,t)=\mathrm{Pr}\{N_{\beta}(t)=0\}=E_{\beta,1}\left(-\alpha(e^{\theta}-1)t^{\beta}\right).
\end{equation*}
We recall that $X_i$'s are independent of $N_{\beta}(t)$ in (\ref{compoundr1}). Thus, for $n\ge1$, we have
\begin{align}\label{rdee32}
q_{\beta}(n,t)&=\sum_{k=1}^{n}\mathrm{Pr}\{X_{1}+X_{2}+\dots+X_{k}=n\}\mathrm{Pr}\{N_{\beta}(t)=k\}\\
&=\sum_{k=1}^{n}\sum_{\Theta_{n}^{k}}k!\prod_{j=1}^{n}\frac{(\theta^{j}/j!)^{x_{j}}}{x_{j}!}\left(\alpha t^{\beta}\right)^{k}E_{\beta,k\beta+1}^{k+1}\left(-\alpha(e^{\theta}-1)t^{\beta}\right),\label{tkn}
\end{align}
where $x_j$ is the total number of claims of $j$ units and 
\begin{equation}\label{tknd}
\Theta_{n}^{k}=\left\{(x_{1},x_{2},\dots,x_{n}):\sum_{j=1}^{n}x_{j}=k,\ \sum_{j=1}^{n}jx_{j}=n,\ x_{j}\in\mathbb{N}\cup \{0\}\right\}.	
\end{equation}
Again as $X_{i}$'s are iid, we have
\begin{align}\label{sweee32}
\mathrm{Pr}\{X_{1}+X_{2}+\dots+X_{k}=n\}&=\underset{m_j\in\mathbb{N}}{\underset{m_{1}+m_{2}+\dots+m_{k}=n}{\sum}}\mathrm{Pr}\{X_{1}=m_1,X_{2}=m_2,\ldots,X_{k}=m_k\}\nonumber\\
&=\underset{m_j\in\mathbb{N}}{\underset{m_{1}+m_{2}+\dots+m_{k}=n}{\sum}}\prod_{j=1}^{k}\mathrm{Pr}\{X_{j}=m_j\}\nonumber\\
&=\underset{m_j\in\mathbb{N}}{\underset{m_{1}+m_{2}+\dots+m_{k}=n}{\sum}}\frac{1}{(e^{\theta}-1)^{k}}\prod_{j=1}^{k}\frac{\theta^{m_{j}}}{m_{j}!},	
\end{align}
where we have used (\ref{ajjagh12}). On substituting (\ref{sweee32}) in (\ref{rdee32}), we get an equivalent expression for the pmf of FBTP in the following form:
\begin{equation}\label{pmf41}
q_{\beta}(n,t)=\sum_{k=1}^{n}\underset{m_j\in\mathbb{N}}{\underset{m_{1}+m_{2}+\dots+m_{k}=n}{\sum}}\prod_{j=1}^{k}\frac{\theta^{m_{j}}}{m_{j}!}\left(\alpha t^{\beta}\right)^{k}E_{\beta,k\beta+1}^{k+1}\left(-\alpha(e^{\theta}-1)t^{\beta}\right).	
\end{equation}

The pmf \eqref{tkn} can be written in the following equivalent form by using Lemma 2.4 of Kataria and Vellaisamy (2017b):
\begin{equation}\label{pm42}
q_{\beta}(n,t)=\sum_{k=1}^{n}\sum_{\Lambda_{n}^{k}}k!\prod_{j=1}^{n-k+1}\frac{(\theta^{j}/j!)^{x_{j}}}{x_{j}!}\left(\alpha t^{\beta}\right)^{k}E_{\beta,k\beta+1}^{k+1}\left(-\alpha(e^{\theta}-1)t^{\beta}\right).
\end{equation}
where 
\begin{equation}\label{Lkn}
	\Lambda_{n}^{k}=\left\{(x_{1},x_{2},\dots,x_{n-k+1}):\sum_{j=1}^{n-k+1}x_{j}=k, \ \  \sum_{j=1}^{n-k+1}jx_{j}=n, \ x_{j}\in\mathbb{N}\cup \{0\}\right\}.
\end{equation}
Thus, we have obtained the five alternate forms of the pmf of FBTP given in (\ref{pmf1}), (\ref{pmf2}), \eqref{tkn}, \eqref{pmf41} and \eqref{pm42}. 

By using (\ref{meanbtp}), (\ref{varbtp}) and Theorem 2.1 of Leonenko {\it et al.} (2014), the mean, variance and  covariance of FBTP can be obtained in the following form:
\begin{align}
\mathbb{E}\left(\mathcal{M}_{\beta}(t)\right)&=\alpha\theta e^{\theta}\mathbb{E}\left(Y_{\beta}(t)\right),\nonumber
\\
\operatorname{Var}\left(\mathcal{M}_{\beta}(t)\right)&=\alpha\theta(\theta+1)e^{\theta}\mathbb{E}\left(Y_{\beta}(t)\right)+\left(\alpha\theta e^{\theta}\right)^{2}\operatorname{Var}\left(Y_{\beta}(t)\right),\label{vs2}\\\operatorname{Cov}\left(\mathcal{M}_{\beta}(s),\mathcal{M}_{\beta}(t)\right)&=\alpha\theta(\theta+1)e^{\theta}\mathbb{E}\left(Y_{\beta}(\min\{s,t\})\right)+\left(\alpha\theta e^{\theta}\right)^{2}\operatorname{Cov}\left(Y_{\beta}(s),Y_{\beta}(t)\right).\label{cs3}
\end{align}
The FBTP exhibits overdispersion as $\operatorname{Var}\left(\mathcal{M}_{\beta}(t)\right)-\mathbb{E}\left(\mathcal{M}_{\beta}(t)\right)>0$ for all $t>0$.

\begin{theorem}\label{thm4.5}
The FBTP exhibits the LRD property.
\end{theorem}
\begin{proof}
From (\ref{vs2}) and (\ref{cs3}), we get
\begin{equation*}
\operatorname{Corr}	\left(\mathcal{M}_{\beta}(s),\mathcal{M}_{\beta}(t)\right)=\frac{\alpha\theta(\theta+1)e^{\theta}\mathbb{E}\left(Y_{\beta}(\min\{s,t\})\right)+(\alpha\theta
e^{\theta})^{2}\operatorname{Cov}\left(Y_{\beta}(s),Y_{\beta}(t)\right)}{\sqrt{\operatorname{Var}\left(\mathcal{M}_{\beta}(s)\right)}\sqrt{\alpha\theta(\theta+1)e^{\theta}\mathbb{E}\left(Y_{\beta}(t)\right)+(\alpha\theta e^{\theta})^{2}\operatorname{Var}\left(Y_{\beta}(t)\right)}}.
\end{equation*}
On using (\ref{meani})-(\ref{asi1}) for fixed $s$ and large $t$, we get
\begin{align*}
\operatorname{Corr}	(\mathcal{M}_{\beta}(s),&\mathcal{M}_{\beta}(t))\\
&\sim \frac{\alpha\theta(\theta+1)e^{\theta}\Gamma^2(\beta+1)\mathbb{E}\left(Y_{\beta}(s)\right)+(\alpha\theta e^{\theta})^{2}\left( \beta s^{2\beta}B(\beta,\beta+1)-\frac{\beta^{2}s^{\beta+1}}{(\beta+1)t^{1-\beta}}\right)}{\Gamma^2(\beta+1)\sqrt{\operatorname{Var}\left(\mathcal{M}_{\beta}(s)\right)}\sqrt{\frac{\alpha\theta(\theta+1)e^{\theta}t^{\beta}}{\Gamma(\beta+1)}+\frac{2(\alpha\theta e^{\theta})^{2}t^{2\beta}}{\Gamma(2\beta+1)}-\frac{(\alpha\theta e^{\theta})^{2}t^{2\beta}}{\Gamma^{2}(\beta+1)}}}\\
&\sim c_{0}(s)t^{-\beta},
\end{align*}
where
\begin{equation*}
c_{0}(s)=	\frac{(\theta+1)\Gamma^2(\beta+1)\mathbb{E}\left(Y_{\beta}(s)\right)+\alpha\theta e^{\theta}\beta s^{2\beta}B(\beta,\beta+1)}{\Gamma^2(\beta+1)\sqrt{\operatorname{Var}\left(\mathcal{M}_{\beta}(s)\right)}\sqrt{\frac{2}{\Gamma(2\beta+1)}-\frac{1}{\Gamma^{2}(\beta+1)}}}.
\end{equation*}
As $0<\beta<1$, it follows that the FBTP has the LRD property.
\end{proof}
\begin{remark}\label{srd}
For a fixed $h>0$, the increment process of FBTP is defined as
\begin{equation*}
\mathcal{Z}_{\beta}^{h}(t)\coloneqq \mathcal{M}_{\beta}(t+h)-\mathcal{M}_{\beta}(t),\ t\ge0.
\end{equation*}
It can be shown that the increment process $\{\mathcal{Z}_{\beta}^{h}(t)\}_{t\ge0}$ exhibits the SRD property. The proof follows similar lines to that of Theorem 1 of Maheshwari and Vellaisamy (2016).
\end{remark}

The factorial moments of the FBTP can be obtained by taking $\lambda_{j}=\alpha \theta^{j}/j!$ for all $j\ge 1$ and letting $k\to \infty$ in Proposition 4 of Kataria and Khandakar (2022).
\begin{proposition}
Let $\psi_{\beta}(r,t)=	\mathbb{E}(\mathcal{M}_{\beta}(t)(\mathcal{M}_{\beta}(t)-1)\cdots(\mathcal{M}_{\beta}(t)-r+1))$, $r\ge1$ be the $r$th factorial moment of FBTP. Then,
\begin{equation*}
\psi_{\beta}(r,t)=r!\sum_{n=1}^{r}\frac{\left(\alpha e^{\theta}t^{\beta}\right)^{n}}{\Gamma(n\beta+1)}\underset{m_i\in\mathbb{N}}{\underset{\sum_{i=1}^nm_i=r}{\sum}}\prod_{\ell=1}^{n}\frac{\theta^{m_\ell}}{m_\ell!}.
\end{equation*}
\end{proposition}
The proof of the next result follows on using \eqref{lemmaar}, the self-similarity property of $\{Y_{\beta}(t)\}_{t\ge0}$ in \eqref{ybt} and the arguments used in Proposition 3 of Kataria and Khandakar (2022).
\begin{proposition}
The one-dimensional distributions of FBTP are not infinitely divisible.
\end{proposition}

\begin{remark}\label{remark3.9}
Let the random variable $\mathcal{W}_1$ be the first waiting time of FBTP. Then, the distribution of $\mathcal{W}_1$ is given by
\begin{equation*}
\mathrm{Pr}\{\mathcal{W}_1>t\}=\mathrm{Pr}\{\mathcal{M}_{\beta}(t)=0\}=E_{\beta, 1}\left(-\alpha (e^{\theta}-1)t^{\beta}\right),
\end{equation*}
which coincides with the first waiting time of TFPP with intensity $\alpha (e^{\theta}-1)$ (see Kataria and Vellaisamy (2017a), Remark 3.3). However,
the one-dimensional distributions of TFPP and FBTP differ. Thus, the fact that the TFPP is a
renewal process (see Meerschaert {\it et al.} (2011)) implies that the FBTP is not a renewal process.
\end{remark}
\section{Poisson-logarithmic process and its fractional version}
Here, we introduce a fractional version of the PLP. First, we give some additional properties of it.

On taking $\beta=1$, $\lambda_{j}=-\lambda (1-p)^{j}/j\ln p$ for all $j\ge 1$ and letting $k\to \infty$, the governing system of differential equations (\ref{cre}) for GCP reduces to the governing system of differential equations  (\ref{btpgov3}) of PLP. Thus, the PLP is a limiting case of the GCP. Thus, we note that several results for PLP can be obtained from the corresponding results for GCP.

The next result gives a recurrence relation for the pmf of PLP that follows from Proposition 1 of Kataria and Khandakar (2022).
 \begin{proposition}
The  state probabilities $\hat{q}(n,t)=\mathrm{Pr}\{\hat{\mathcal{M}}(t)=n\}$, $n\ge1$ of PLP satisfy
\begin{equation*}
\hat{q}(n,t)=-\frac{\lambda t}{n\ln p}\sum_{j=1}^{n}(1-p)^{j}\hat{q}(n-j,t).	
\end{equation*}
\end{proposition}

The pgf \eqref{pgfplp} can be rewritten as
\begin{equation*}
\hat{G}(u,t)=\prod_{j=1}^{\infty}\exp \left(-\frac{\lambda t}{\ln p} \frac{(1-p)^{j}}{j}(u^{j}-1)\right).
\end{equation*}
It follows that the PLP is equal in distribution to a weighted sum of independent Poisson processes, that is,
\begin{equation}\label{rela}
\hat{\mathcal{M}}(t)\stackrel{d}{=}\sum_{j=1}^{\infty}jN_{j}(t),
\end{equation}	
where $\{N_{j}(t)\}_{t\ge0}$ is a Poisson process with intensity $\lambda_{j}=-\lambda (1-p)^{j}/j\ln p$. Thus, 
\begin{align}\label{limit plp}
\lim_{t\to\infty}\frac{\hat{\mathcal{M}}(t)}{t}&\stackrel{d}{=}\sum_{j=1}^{\infty}j\lim_{t\to\infty}\frac{N_{j}(t)}{t}\nonumber\\
&=\frac{\lambda(p-1)}{p\ln p},\ \text{in probability},
\end{align}
where we have used $\lim_{t\to\infty}N_{j}(t)/t=-\lambda (1-p)^{j}/j\ln p$, $j\ge1$ a.s. 
	
In view of \eqref{janina} and \eqref{rela}, the pmf of PLP is given by
\begin{equation}\label{p(n,t)13}
\hat{q}(n,t)=\sum_{\Omega_{n}}\prod_{j=1}^{n}\frac{\left((1-p)^{j}/j\right)^{x_{j}}}{x_{j}!}\left(\frac{-\lambda t}{\ln p}\right)^{z_{n}}e^{-\lambda t}, 
\end{equation}
where $z_{n}=x_{1}+x_{2}+\dots+x_{n}$ and $\Omega_{n}$ is given in \eqref{onn}. 

Using \eqref{manina}, the pmf of PLP can alternatively be written as 
\begin{equation}\label{alterpnt3}
\hat{q}(n,t)=\sum_{\Omega^{n}_{k}}\prod_{j=1}^{k}\frac{(1-p)^{x_{j}}}{x_{j}k!}\left(-\frac{\lambda t}{\ln p}\right)^{k}e^{-\lambda t},
\end{equation}
where $\Omega^{n}_{k}$ is given in \eqref{okn}.

The next result gives a martingale characterization for the PLP whose proof follows from Proposition 2 of Kataria and Khandakar (2022).
 \begin{proposition}
The process $\left\{\hat{\mathcal{M}}(t)-\dfrac{\lambda(p-1)}{p\ln p}t\right\}_{t\geq0}$ is a martingale with respect to a natural filtration $\mathscr{F}_{t}=\sigma\left(\hat{\mathcal{M}}(s), s\le t\right)$.
\end{proposition} 

Let $\hat{r}_{1}=\lambda(p-1)/p\ln p$ and $\hat{r}_{2}=\lambda(p-1)/p^{2}\ln p$. From \eqref{plp1}, it follows that the PLP is a L\'evy process. Its L\'evy measure can be obtained by taking $\lambda_{j}=-\lambda (1-p)^{j}/j\ln p$ for all $j\ge 1$ and letting $k\to \infty$ in Eq. (13) of Kataria and Khandakar (2022). It is given by 
\begin{equation*}
\hat{\mu}(\mathrm{d}x)=-\frac{\lambda}{\ln p}\sum_{j=1}^{\infty}\frac{(1-p)^{j}}{j}\delta_{j}\mathrm{d}x,
\end{equation*}
where $\delta_{j}$'s are Dirac measures. Its mean, variance and covariance are given by
\begin{equation}\label{meanplp}
E(\hat{\mathcal{M}}(t))=\hat{r}_{1}t, \ \operatorname{Var}(\hat{\mathcal{M}}(t))=\hat{r}_{2}t,\ 
\operatorname{Cov}(\hat{\mathcal{M}}(s),\hat{\mathcal{M}}(t))=\hat{r}_{2}\min\{s,t\}.
\end{equation}

\subsection{Fractional Poisson-logarithmic process}
Here, we introduce a fractional version of the PLP, namely, the fractional Poisson-logarithmic process (FPLP). We define it as the stochastic process  $\{\hat{\mathcal{M}}_{\beta}(t)\}_{t\ge0}$, $0<\beta\le 1$,  whose state probabilities  $\hat{q}_{\beta}(n,t)=\mathrm{Pr}\{\hat{\mathcal{M}}_{\beta}(t)=n\}$ satisfy the following system of differential equations: 
\begin{equation}\label{fbtpgov3}
\begin{aligned}
\frac{\mathrm{d}^{\beta}}{\mathrm{d}t^{\beta}}\hat{q}_{\beta}(0,t)&=-\lambda \hat{q}_{\beta}(0,t),\\
\frac{\mathrm{d}^{\beta}}{\mathrm{d}t^{\beta}}\hat{q}_{\beta}(n,t)&=-\lambda \hat{q}_{\beta}(n,t)-\frac{\lambda}{\ln p}\sum_{j=1}^{n}\frac{(1-p)^{j}}{j}\hat{q}_{\beta}(n-j,t),\ n\ge 1,
\end{aligned}
\end{equation}
 with initial conditions $\hat{q}_{\beta}(0,0)=1$ and $\hat{q}_{\beta}(n,0)=0$, $n\ge1$.
 
Note that the system of equations (\ref{fbtpgov3}) is obtained by replacing the integer order derivative in (\ref{btpgov3}) by Caputo fractional derivative.
 
\begin{remark}
On taking $\lambda_{j}=-\lambda (1-p)^{j}/j\ln p$ for all $j\ge 1$ and letting $k\to \infty$, the System (\ref{cre}) reduces to the System (\ref{fbtpgov3}). Thus, the FPLP is a limiting case of the GFCP.
	
Also, on taking $\lambda=-\ln p$, the System \eqref{fbtpgov3} reduces to the system of differential equations that governs the state probabilities of fractional negative binomial process (see Beghin (2015), Eq. (66)).
 \end{remark}
 Using \eqref{fbtpgov3}, it can be shown that the pgf $\hat{G}_{\beta}(u,t)=\mathbb{E}\left(u^{\hat{\mathcal{M}}_{\beta}(t)}\right)$, $|u|\leq1$ of FPLP satisfies
\begin{equation*}
\frac{\mathrm{d}^{\beta}}{\mathrm{d}t^{\beta}}\hat{G}_{\beta}(u,t)=-\lambda \left(1-\frac{\ln(1-(1-p)u)}{\ln p}\right)\hat{G}_{\beta}(u,t),
\end{equation*}
with $\hat{G}_{\beta}(u,0)=1$. On taking the Laplace transform in the above equation and using \eqref{lc}, we get
\begin{equation*}
s^{\beta}\tilde{\hat{G}}_{\beta}(u,s)-s^{\beta-1}\hat{G}_{\beta}(u,0)=-\lambda \left(1-\frac{\ln(1-(1-p)u)}{\ln p}\right)\tilde{\hat{G}}_{\beta}(u,s),\ s>0.
\end{equation*}
Thus, 
\begin{equation*}
\tilde{\hat{G}}_{\beta}(u,s)=\frac{s^{\beta-1}}{s^{\beta}+\lambda \left(1-\frac{\ln(1-(1-p)u)}{\ln p}\right)}.
\end{equation*}
On taking inverse Laplace transform and using (\ref{mi}), we get
\begin{equation}\label{pgfftbp3}
\hat{G}_{\beta}(u,t)=	E_{\beta,1}\left(-\lambda \left(1-\frac{\ln(1-(1-p)u)}{\ln p}\right)t^{\beta}\right).
\end{equation}
\begin{remark}
On taking $\beta=1$ in \eqref{pgfftbp3}, we get the pgf of PLP given in (\ref{pgfplp}). Also, from \eqref{pgfftbp3} we can verify that the pmf $\hat{q}_{\beta}(n,t)$ sums up to one, that is, $\sum_{n=0}^{\infty}\hat{q}_{\beta}(n,t)=\hat{G}_{\beta}(u,t)|_{u=1}=E_{\beta,1}(0)=1$.
\end{remark}
The next result gives a time-changed relationship between the PLP and its fractional variant, FPLP.
\begin{theorem}\label{thm4.13}
Let $\{Y_{\beta}(t)\}_{t\ge0}$, $0<\beta<1$, be an inverse stable subordinator independent of the PLP $\{\hat{\mathcal{M}}(t)\}_{t\ge0}$. Then
\begin{equation}\label{ybt3}
\hat{\mathcal{M}}_{\beta}(t)\stackrel{d}{=}\hat{\mathcal{M}}(Y_{\beta}(t)),\ t\ge0.
\end{equation}
\end{theorem}
The proof of Theorem \ref{thm4.13} follows similar lines to that of Theorem \ref{thm4.1}.
\begin{remark}
In view of Remark \ref{remark3.5}, we have
\begin{equation}\label{keyyekk31}
\hat{\mathcal{M}}_{\beta}(t)\overset{d}{=}\hat{\mathcal{M}}(T_{2\beta}(t)),\ t>0,
\end{equation}	
where $\{T_{2\beta}(t)\}_{t>0}$ is independent of $\{\hat{\mathcal{M}}(t)\}_{t>0}$.
\end{remark}

\begin{remark}
In view of \eqref{plp1} and \eqref{ybt3}, we note that the FPLP is equal in distribution to the following compound fractional Poisson process:
\begin{equation}\label{compoundr3}
\hat{\mathcal{M}}_{\beta}(t)\stackrel{d}{=}\sum_{i=1}^{N_{\beta}(t)}X_{i},\ t\ge0,
\end{equation}
where $\{N_{\beta}(t)\}_{t\ge0}$ is a TFPP with intensity $\lambda$ independent of the sequence of iid random variables $\{X_{i}\}_{i\ge1}$. Therefore, it is neither Markovian nor a L\'evy process. In view of \eqref{compoundr3}, the System \eqref{fbtpgov3} can alternatively be obtained using Proposition 1 of Beghin and Macci (2014).
\end{remark}
	
\begin{theorem}\label{thm4.23}
The pmf $\hat{q}_{\beta}(n,t)=\mathrm{Pr}\{\hat{\mathcal{M}}_{\beta}(t)=n\}$ of FPLP is given by
\begin{equation}\label{pmf}
\hat{q}_{\beta}(n,t)=\begin{cases*}
E_{\beta, 1}\left(-\lambda t^{\beta}\right),\ n=0,\\
\displaystyle\sum_{\Omega_{n}}\prod_{j=1}^{n}\frac{\left((1-p)^{j}/j\right)^{x_{j}}}{x_{j}!}z_{n}!\left(\frac{-\lambda t^{\beta}}{\ln p}\right)^{z_{n}}E_{\beta, z_{n}\beta +1}^{z_{n}+1}\left(-\lambda t^{\beta}\right),\ n\ge1,
\end{cases*}
\end{equation}
where  $z_{n}=x_{1}+x_{2}+\dots+x_{n}$ and $\Omega_{n}$ is given in \eqref{onn}.
\end{theorem}
\begin{proof}
From \eqref{ub} and \eqref{keyyekk31}, we have
\begin{equation*}
\hat{q}_{\beta}(n,t)=\int_{0}^{\infty}\hat{q}(n,x)\bar{u}_{2\beta}(x,t)\mathrm{d}x,
\end{equation*}
where $\hat{q}(n,x)$ is given in \eqref{p(n,t)13}. From this point the proof follows similar lines to that of Theorem \ref{thm4.21}.
\end{proof}

\begin{remark}
If we use \eqref{alterpnt3} in the proof Theorem \ref{thm4.23} then we can obtain the following alternate form of the pmf of FPLP:
\begin{equation}\label{pmf2p}
\hat{q}_{\beta}(n,t)=\sum_{\Omega^{n}_{k}}\prod_{j=1}^{k}\frac{(1-p)^{x_{j}}}{x_{j}}\left(\frac{-\lambda t^{\beta}}{\ln p}\right)^{k} E_{\beta, k\beta+1}^{k+1}\left(-\lambda t^{\beta}\right),
\end{equation}	
where $\Omega^{n}_{k}$ is given in \eqref{okn}.
\end{remark}
From (\ref{compoundr3}), we have
\begin{equation*}
\hat{q}_{\beta}(0,t)=\mathrm{Pr}\{N_{\beta}(t)=0\}=E_{\beta,1}(-\lambda t^{\beta}).
\end{equation*}
For $n\ge1$, we get
\begin{align}\label{rdee323}
\hat{q}_{\beta}(n,t)&=\sum_{k=1}^{n}\mathrm{Pr}\{X_{1}+X_{2}+\dots+X_{k}=n\}\mathrm{Pr}\{N_{\beta}(t)=k\}\\
&=\sum_{k=1}^{n}\sum_{\Theta_{n}^{k}}k!\prod_{j=1}^{n}\frac{((1-p)^{j}/j)^{x_{j}}}{x_{j}!}\left(\frac{-\lambda t^{\beta}}{\ln p}\right)^{k}E_{\beta,k\beta+1}^{k+1}(-\lambda t^{\beta}),\label{tkn3}
\end{align}
where $\Theta_{n}^{k}$ is given in (\ref{tknd}).
	
As $X_{i}$'s are iid, we have
\begin{align}\label{sweee323}
\mathrm{Pr}\{X_{1}+X_{2}+\dots+X_{k}=n\}
&=\underset{m_j\in\mathbb{N}}{\underset{m_{1}+m_{2}+\dots+m_{k}=n}{\sum}}\prod_{j=1}^{k}\mathrm{Pr}\{X_{j}=m_j\}\nonumber\\
&=\underset{m_j\in\mathbb{N}}{\underset{m_{1}+m_{2}+\dots+m_{k}=n}{\sum}}\left(\frac{-1}{\ln p}\right)^{k}\prod_{j=1}^{k}\frac{(1-p)^{m_{j}}}{m_{j}
},	
\end{align}
where we have used (\ref{plp1}). Substituting (\ref{sweee323}) in (\ref{rdee323}), we get an equivalent expression for the pmf of the FPLP in the following form:
\begin{equation}\label{pmf4p}
\hat{q}_{\beta}(n,t)=\sum_{k=1}^{n}\underset{m_j\in\mathbb{N}}{\underset{m_{1}+m_{2}+\dots+m_{k}=n}{\sum}}\prod_{j=1}^{k}\frac{(1-p)^{m_{j}}}{m_{j}!}\left(\frac{-\lambda t^{\beta}}{\ln p}\right)^{k}E_{\beta,k\beta+1}^{k+1}(-\lambda t^{\beta}).	
\end{equation}
 
The pmf \eqref{tkn3} can be written in the following equivalent form by using Lemma 2.4 of Kataria and Vellaisamy (2017b): 
\begin{equation}\label{pm42p}
\hat{q}_{\beta}(n,t)=\sum_{k=1}^{n}\sum_{\Lambda_{n}^{k}}k!\prod_{j=1}^{n-k+1}\frac{((1-p)^{j}/j)^{x_{j}}}{x_{j}!}\left(\frac{-\lambda t^{\beta}}{\ln p}\right)^{k}E_{\beta,k\beta+1}^{k+1}(-\lambda t^{\beta}),
\end{equation}
where $\Lambda_{n}^{k}$ is given in \eqref{Lkn}.
	
Thus, we have obtained the five alternate forms of the pmf of FPLP given in (\ref{pmf}), (\ref{pmf2p}), \eqref{tkn3}, \eqref{pmf4p} and \eqref{pm42p}.
\begin{remark}
The distribution of the first waiting time $\hat{\mathcal{W}}_{1}$ of FPLP is given by
\begin{equation*}
\mathrm{Pr}\{\hat{\mathcal{W}}_{1}>t\}=\mathrm{Pr}\{\hat{\mathcal{M}}_{\beta}(t)=0\}=E_{\beta, 1}\left(-\lambda t^{\beta}\right).
\end{equation*}
By using the same arguments as used in Remark \ref{remark3.9}, we conclude that the FPLP is not a renewal process. 
\end{remark}
By using \eqref{meanplp} and Theorem 2.1 of Leonenko {\it et al.} (2014),  we obtain the mean, variance and  covariance of FPLP as follows: 
\begin{align*}
\mathbb{E}\left(\hat{\mathcal{M}}_{\beta}(t)\right)&=\hat{r}_{1}\mathbb{E}\left(Y_{\beta}(t)\right),
\\
\operatorname{Var}\left(\hat{\mathcal{M}}_{\beta}(t)\right)&=\hat{r}_{2}\mathbb{E}\left(Y_{\beta}(t)\right)+\hat{r}_{1}^{2}\operatorname{Var}\left(Y_{\beta}(t)\right),\\	
\operatorname{Cov}\left(\hat{\mathcal{M}}_{\beta}(s),\hat{\mathcal{M}}_{\beta}(t)\right)&=\hat{r}_{2}\mathbb{E}\left(Y_{\beta}(\min\{s,t\})\right)+\hat{r}_{1}^{2}\operatorname{Cov}\left(Y_{\beta}(s),Y_{\beta}(t)\right).
\end{align*}
The FPLP exhibits overdispersion as $\operatorname{Var}\left(\hat{\mathcal{M}}_{\beta}(t)\right)-\mathbb{E}\left(\hat{\mathcal{M}}_{\beta}(t)\right)>0$ for all $t>0$.
\begin{theorem}\label{thm4.531}
The FPLP exhibits the LRD property.
\end{theorem}
The proof of Theorem \ref{thm4.531}  follows similar lines to that of Theorem \ref{thm4.5}.
\begin{remark}
As in Remark \ref{srd}, for a fixed $h>0$, the increment process $\hat{\mathcal{Z}}_{\beta}^{h}(t)\coloneqq \hat{\mathcal{M}}_{\beta}(t+h)-\hat{\mathcal{M}}_{\beta}(t)$, $t\ge0$ of FPLP
exhibits the SRD property.		
\end{remark}
\begin{proposition}\label{prop}
The one-dimensional distributions of FPLP are not infinitely divisible.
\end{proposition}
The proof of Proposition \ref{prop} follows on using \eqref{limit plp}, the self-similarity property of $\{Y_{\beta}(t)\}_{t\ge0}$ in \eqref{ybt3} and the arguments used in Proposition 3 of Kataria and Khandakar (2022).

\section{Generalized P\'olya-Aeppli Process and its fractional version}
Here, we introduce a fractional version of the GPAP. First, we give some additional properties of it.

On taking $\beta=1$, $\lambda_{j}=\displaystyle\lambda\binom{r+j-1}{j}\rho^{j}(1-\rho)^{r}/(1-(1-\rho)^{r})$ for all $j\ge 1$ and letting $k\to \infty$, the governing system of differential equations (\ref{cre}) for GCP reduces to the governing system of differential equations  (\ref{btpgov33}) of GPAP. Thus, the GPAP is a limiting case of the GCP.	
	
The next result gives a recurrence relation for the pmf of GPAP whose proof follows from Proposition 1 of Kataria and Khandakar (2022).
\begin{proposition}
The  state probabilities $\bar{q}(n,t)=\mathrm{Pr}\{\bar{\mathcal{M}}(t)=n\}$, $n\ge1$ of GPAP satisfy
\begin{equation*}
\bar{q}(n,t)=\frac{\lambda t(1-\rho)^{r}}{n(1-(1-\rho)^{r})}\sum_{j=1}^{n}j\rho^{j}\binom{r+j-1}{j}\bar{q}(n-j,t).	
\end{equation*}
\end{proposition}
	
The pgf \eqref{pgfplp3} can be expressed as
\begin{equation*}
\bar{G}(u,t)=\prod_{j=1}^{\infty}\exp \left(\lambda t\frac{\rho^{j}(1-\rho)^{r}}{1-(1-\rho)^{r}}\binom{r+j-1}{j}(u^{j}-1)\right).
\end{equation*}
It follows that the GPAP is equal in distribution to a weighted sum of independent Poisson processes, that is,
\begin{equation}\label{relat}
\bar{\mathcal{M}}(t)\stackrel{d}{=}\sum_{j=1}^{\infty}jN_{j}(t),
\end{equation}	
where $\{N_{j}(t)\}_{t\ge0}$ is a Poisson process with intensity $\lambda_{j}=\displaystyle\lambda\binom{r+j-1}{j}\rho^{j}(1-\rho)^{r}/(1-(1-\rho)^{r})$. Thus, 
\begin{align}\label{limit}
\lim_{t\to\infty}\frac{\bar{\mathcal{M}}(t)}{t}&\stackrel{d}{=}\sum_{j=1}^{\infty}j\lim_{t\to\infty}\frac{N_{j}(t)}{t}\nonumber\\
&=\frac{\lambda r\rho}{(1-\rho)(1-(1-\rho)^{r})},\ \text{in probability},
\end{align}
where we have used $\lim_{t\to\infty}N_{j}(t)/t=\displaystyle\lambda\binom{r+j-1}{j}\rho^{j}(1-\rho)^{r}/(1-(1-\rho)^{r})$, $j\ge1$ a.s. On substituting $r=1$ in \eqref{limit}, we get the corresponding limiting result for the P\'olya-Aeppli process (see Kataria and Khandakar (2022), Section 4.3).
	
In view of \eqref{relat}, the pmf of GPAP can be obtained from \eqref{janina} and it is given by
\begin{equation}\label{p(n,t)131}
\bar{q}(n,t)=\sum_{\Omega_{n}}\prod_{j=1}^{n}\frac{\left(\rho^{j}\binom{r+j-1}{j}\right)^{x_{j}}}{x_{j}!}\left(\frac{\lambda t (1-\rho)^{r}}{1-(1-\rho)^{r}}\right)^{z_{n}}e^{-\lambda t}, 
\end{equation}
where $z_{n}=x_{1}+x_{2}+\dots+x_{n}$ and $\Omega_{n}$ is given in \eqref{onn}. 
	
Using \eqref{manina}, the pmf of GPAP can alternatively be written as 
\begin{equation}\label{alterpnt31}
\bar{q}(n,t)=\sum_{\Omega^{n}_{k}}\prod_{j=1}^{k}\rho^{x_{j}}\binom{r+x_{j}-1}{x_{j}}\left(\frac{\lambda t(1-\rho)^{r}}{1-(1-\rho)^{r}}\right)^{k}\frac{e^{-\lambda t}}{k!},
\end{equation}
where $\Omega^{n}_{k}$ is given in \eqref{okn}.
	
From \eqref{gpap1}, it follows that the GPAP is a L\'evy process.
Its L\'evy measure can be obtained by taking $\lambda_{j}=\displaystyle\lambda\binom{r+j-1}{j}\rho^{j}(1-\rho)^{r}/(1-(1-\rho)^{r})$ for all $j\ge 1$ and letting $k\to \infty$ in Eq. (13) of Kataria and Khandakar (2022). It is given by
\begin{equation*}
\bar{\mu}(\mathrm{d}x)=\frac{\lambda(1-\rho)^{r}}{1-(1-\rho)^{r}}\sum_{j=1}^{\infty}\binom{r+j-1}{j}\rho^{j}\delta_{j}\mathrm{d}x.
\end{equation*}
\subsection{Fractional generalized P\'olya-Aeppli process}
Here, we introduce a fractional version of the GPAP, namely, the fractional generalized P\'olya-Aeppli process (FGPAP). We define it as the stochastic process  $\{\bar{\mathcal{M}}_{\beta}(t)\}_{t\ge0}$, $0<\beta\le 1$,  whose state probabilities  $\bar{q}_{\beta}(n,t)=\mathrm{Pr}\{\bar{\mathcal{M}}_{\beta}(t)=n\}$ satisfy the following system of differential equations: 
\begin{equation}\label{fbtpgov31}
\begin{aligned}
\frac{\mathrm{d}^{\beta}}{\mathrm{d}t^{\beta}}\bar{q}_{\beta}(0,t)&=-\lambda \bar{q}_{\beta}(0,t),\\
\frac{\mathrm{d}^{\beta}}{\mathrm{d}t^{\beta}}\bar{q}_{\beta}(n,t)&=-\lambda \bar{q}_{\beta}(n,t)+\frac{\lambda}{(1-\rho)^{-r}-1}\sum_{j=1}^{n}\binom{r+j-1}{j}\rho^{j}\bar{q}_{\beta}(n-j,t),\ n\ge 1,
\end{aligned}
\end{equation}
with initial conditions $\bar{q}_{\beta}(0,0)=1$ and $\bar{q}_{\beta}(n,0)=0$, $n\ge1$.
	
Note that the system of equations (\ref{fbtpgov31}) is obtained by replacing the integer order derivative in (\ref{btpgov33}) by Caputo fractional derivative.
	
\begin{remark}
On taking $\lambda_{j}=\displaystyle\lambda\binom{r+j-1}{j}\rho^{j}(1-\rho)^{r}/(1-(1-\rho)^{r})$ for all $j\ge 1$ and letting $k\to \infty$, the System (\ref{cre}) reduces to the System (\ref{fbtpgov31}). Thus, the FGPAP is a limiting case of the GFCP.
		
Also, on taking $r=1$, the System \eqref{fbtpgov31} reduces to the system of differential equations that governs the state probabilities of fractional P\'olya-Aeppli process (see Beghin and Macci (2014), Eq. (19)).
\end{remark}
Using \eqref{fbtpgov31}, it can be shown that the pgf $\bar{G}_{\beta}(u,t)=\mathbb{E}\left(u^{\bar{\mathcal{M}}_{\beta}(t)}\right)$, $|u|\leq1$ of FGPAP satisfies
\begin{equation*}
\frac{\mathrm{d}^{\beta}}{\mathrm{d}t^{\beta}}\bar{G}_{\beta}(u,t)=-\lambda \left(1-\frac{(1-\rho u)^{-r}-1}{(1-\rho)^{-r}-1}\right)\bar{G}_{\beta}(u,t),
\end{equation*}
with $\bar{G}_{\beta}(u,0)=1$. On taking the Laplace transform in the above equation and using \eqref{lc}, we get
\begin{equation*}
s^{\beta}\tilde{\bar{G}}_{\beta}(u,s)-s^{\beta-1}\bar{G}_{\beta}(u,0)=-\lambda \left(1-\frac{(1-\rho u)^{-r}-1}{(1-\rho)^{-r}-1}\right)\tilde{\bar{G}}_{\beta}(u,s),\ s>0.
\end{equation*}
Thus, 
\begin{equation*}
\tilde{\bar{G}}_{\beta}(u,s)=\frac{s^{\beta-1}}{s^{\beta}+\lambda \left(1-\frac{(1-\rho u)^{-r}-1}{(1-\rho)^{-r}-1}\right)}.
\end{equation*}
On taking inverse Laplace transform and using (\ref{mi}), we get
\begin{equation}\label{pgfftbp31}
\bar{G}_{\beta}(u,t)=	E_{\beta,1}\left(-\lambda \left(1-\frac{(1-\rho u)^{-r}-1}{(1-\rho)^{-r}-1}\right)t^{\beta}\right).
\end{equation}
\begin{remark}
On taking $\beta=1$ in \eqref{pgfftbp31}, we get the pgf of GPAP given in (\ref{pgfplp3}). Also, from \eqref{pgfftbp31} we can verify that the pmf $\bar{q}_{\beta}(n,t)$ sums up to one, that is, $\sum_{n=0}^{\infty}\bar{q}_{\beta}(n,t)=\bar{G}_{\beta}(u,t)|_{u=1}=E_{\beta,1}(0)=1$.
\end{remark}
The next result gives a time-changed relationship between the GPAP and its fractional version, FGPAP.
\begin{theorem}\label{thm4.131}
Let $\{Y_{\beta}(t)\}_{t\ge0}$, $0<\beta<1$, be an inverse stable subordinator independent of the GPAP $\{\bar{\mathcal{M}}(t)\}_{t\ge0}$. Then
\begin{equation}\label{ybt31}
\bar{\mathcal{M}}_{\beta}(t)\stackrel{d}{=}\bar{\mathcal{M}}(Y_{\beta}(t)),\ t\ge0.
\end{equation}
\end{theorem}
The proof of Theorem \ref{thm4.131} follows similar lines to that of Theorem \ref{thm4.1}.
\begin{remark}
In view of Remark \ref{remark3.5}, we have
\begin{equation}\label{keyyekk311}
\bar{\mathcal{M}}_{\beta}(t)\overset{d}{=}\bar{\mathcal{M}}(T_{2\beta}(t)),\ t>0,
\end{equation}	
where $\{T_{2\beta}(t)\}_{t>0}$ is independent of $\{\bar{\mathcal{M}}(t)\}_{t>0}$.
\end{remark}
\begin{remark}
In view of \eqref{gpap1} and \eqref{ybt31}, we note that the FGPAP is equal in distribution to the following compound fractional Poisson process:
\begin{equation}\label{compoundr31}
\bar{\mathcal{M}}_{\beta}(t)\stackrel{d}{=}\sum_{i=1}^{N_{\beta}(t)}X_{i},\ t\ge0,
\end{equation}
where $\{N_{\beta}(t)\}_{t\ge0}$ is a TFPP with intensity $\lambda$ independent of $\{X_{i}\}_{i\ge1}$. Therefore, it is neither Markovian nor a L\'evy process. In view of \eqref{compoundr31}, the System \eqref{fbtpgov31} can alternatively be obtained using Proposition 1 of Beghin and Macci (2014).
\end{remark}
	
\begin{theorem}\label{thm4.231}
The pmf $\bar{q}_{\beta}(n,t)=\mathrm{Pr}\{\bar{\mathcal{M}}_{\beta}(t)=n\}$ of FGPAP is given by
\begin{equation}\label{pmf11}
\bar{q}_{\beta}(n,t)=\begin{cases*}
E_{\beta, 1}\left(-\lambda t^{\beta}\right),\ n=0,\\
\displaystyle\rho^{n}\sum_{j=1}^{n}\sum_{m=1}^{j}(-1)^{m}\binom{j}{m}\left(\frac{-\lambda t^{\beta }}{(1-\rho)^{-r}-1}\right)^{j}\binom{rm+n-1}{n}E_{\beta, j\beta +1}^{j+1}\left(-\lambda t^{\beta}\right),\ n\ge1.
\end{cases*}
\end{equation}
\end{theorem}
\begin{proof}
From \eqref{ub} and \eqref{keyyekk311}, we have
\begin{equation*}
\bar{q}_{\beta}(n,t)=\int_{0}^{\infty}\bar{q}(n,x)\bar{u}_{2\beta}(x,t)\mathrm{d}x,
\end{equation*}
where $\bar{q}(n,x)$ is the pmf of GPAP (see Jacob and Jose (2018), Eq. (3)). From this point the proof follows similar lines to that of Theorem \ref{thm4.21}.
\end{proof}
	
\begin{remark}
If we use \eqref{p(n,t)131} in the proof of Theorem \ref{thm4.231} then we can obtain the following alternate form of the pmf of FGPAP:		
\begin{equation}\label{pmf111}
\bar{q}_{\beta}(n,t)=\begin{cases*}
E_{\beta, 1}\left(-\lambda t^{\beta}\right),\ n=0,\\
\displaystyle\sum_{\Omega_{n}}\prod_{j=1}^{n}\frac{\left(\rho^{j}\binom{r+j-1}{j}\right)^{x_{j}}}{x_{j}!}z_{n}!\left(\frac{\lambda t^{\beta}(1-\rho)^{r}}{1-(1-\rho)^{r}}\right)^{z_{n}}E_{\beta, z_{n}\beta +1}^{z_{n}+1}\left(-\lambda t^{\beta}\right),\ n\ge1.
\end{cases*}
\end{equation}
Again,	if we use \eqref{alterpnt31} in the proof of Theorem \ref{thm4.231} then we can obtain the following: 
\begin{equation}\label{pmf2p1}
\bar{q}_{\beta}(n,t)=\sum_{\Omega^{n}_{k}}\prod_{j=1}^{k}\rho^{x_{j}}\binom{r+x_{j}-1}{x_{j}}\left(\frac{\lambda t^{\beta}(1-\rho)^{r}}{1-(1-\rho)^{r}}\right)^{k} E_{\beta, k\beta+1}^{k+1}\left(-\lambda t^{\beta}\right).
\end{equation}	
\end{remark}
From (\ref{compoundr31}), we have
\begin{equation*}
\bar{q}_{\beta}(0,t)=\mathrm{Pr}\{N_{\beta}(t)=0\}=E_{\beta,1}(-\lambda t^{\beta}).
\end{equation*}
For $n\ge1$, we get
\begin{align}\label{rdee3231}
\bar{q}_{\beta}(n,t)&=\sum_{k=1}^{n}\mathrm{Pr}\{X_{1}+X_{2}+\dots+X_{k}=n\}\mathrm{Pr}\{N_{\beta}(t)=k\}\\
&=\sum_{k=1}^{n}\sum_{\Theta_{n}^{k}}k!\prod_{j=1}^{n}\frac{\left(\rho^{j}\binom{r+j-1}{j}\right)^{x_{j}}}{x_{j}!}\left(\frac{\lambda t^{\beta}(1-\rho)^{r}}{1-(1-\rho)^{r}}\right)^{k} E_{\beta,k\beta+1}^{k+1}(-\lambda t^{\beta}),\label{tkn31}
\end{align}
where $\Theta_{n}^{k}$ is given in (\ref{tknd}).
	
As $X_{i}$'s are iid, we have
\begin{align}\label{sweee3231}
\mathrm{Pr}\{X_{1}+X_{2}+\dots+X_{k}=n\}
&=\underset{m_j\in\mathbb{N}}{\underset{m_{1}+m_{2}+\dots+m_{k}=n}{\sum}}\prod_{j=1}^{k}\mathrm{Pr}\{X_{j}=m_j\}\nonumber\\
&=\underset{m_j\in\mathbb{N}}{\underset{m_{1}+m_{2}+\dots+m_{k}=n}{\sum}}\left(\frac{(1-\rho)^{r}}{1-(1-\rho)^{r}}\right)^{k}\prod_{j=1}^{k}\rho^{m_{j}}\binom{r+m_{j}-1}{m_{j}},	
\end{align}
where we have used (\ref{gpap1}). Substituting (\ref{sweee3231}) in (\ref{rdee3231}), we get an equivalent expression for the pmf of FGPAP in the following form:
\begin{equation}\label{pmf4p1}
\bar{q}_{\beta}(n,t)=\sum_{k=1}^{n}\underset{m_j\in\mathbb{N}}{\underset{m_{1}+m_{2}+\dots+m_{k}=n}{\sum}}\left(\frac{\lambda t^{\beta}(1-\rho)^{r}}{1-(1-\rho)^{r}}\right)^{k}\prod_{j=1}^{k}\rho^{m_{j}}\binom{r+m_{j}-1}{m_{j}}E_{\beta,k\beta+1}^{k+1}(-\lambda t^{\beta}).	
\end{equation}
	
The pmf \eqref{tkn31} can be written in the following equivalent form by using Lemma 2.4 of Kataria and Vellaisamy (2017b): 
\begin{equation}\label{pm42p1}
\bar{q}_{\beta}(n,t)=\sum_{k=1}^{n}\sum_{\Lambda_{n}^{k}}k!\prod_{j=1}^{n-k+1}\frac{\left(\rho^{j}\binom{r+j-1}{j}\right)^{x_{j}}}{x_{j}!}\left(\frac{\lambda t^{\beta}(1-\rho)^{r}}{1-(1-\rho)^{r}}\right)^{k}E_{\beta,k\beta+1}^{k+1}(-\lambda t^{\beta}),
\end{equation}
where $\Lambda_{n}^{k}$ is given in \eqref{Lkn}.
	
Thus, we have obtained the six alternate forms of the pmf of FGPAP given in (\ref{pmf11})- \eqref{pmf2p1}, \eqref{tkn31}, \eqref{pmf4p1} and \eqref{pm42p1}. On substituting $r=1$ in these pmfs, we get the equivalent versions of the pmf of fractional P\'olya-Aeppli process.
	
\begin{remark}
The distribution of the first waiting time $\bar{\mathcal{W}}_{1}$ of FGPAP is given by
\begin{equation*}
\mathrm{Pr}\{\bar{\mathcal{W}}_{1}>t\}=\mathrm{Pr}\{\bar{\mathcal{M}}_{\beta}(t)=0\}=E_{\beta, 1}\left(-\lambda t^{\beta}\right).
\end{equation*}
By using the same arguments as used in Remark \ref{remark3.9}, we conclude that the FGPAP is not a renewal process. 
\end{remark}
By using \eqref{meanvar} and Theorem 2.1 of Leonenko {\it et al.} (2014),  we obtain the mean, variance and  covariance of FGPAP as follows: 
\begin{align*}
\mathbb{E}\left(\bar{\mathcal{M}}_{\beta}(t)\right)&=\bar{r}_{1}\mathbb{E}\left(Y_{\beta}(t)\right),\nonumber
\\
\operatorname{Var}\left(\bar{\mathcal{M}}_{\beta}(t)\right)&=\bar{r}_{2}\mathbb{E}\left(Y_{\beta}(t)\right)+\bar{r}_{1}^{2}\operatorname{Var}\left(Y_{\beta}(t)\right),\\	
\operatorname{Cov}\left(\bar{\mathcal{M}}_{\beta}(s),\bar{\mathcal{M}}_{\beta}(t)\right)&=\bar{r}_{2}\mathbb{E}\left(Y_{\beta}(\min\{s,t\})\right)+\bar{r}_{1}^{2}\operatorname{Cov}\left(Y_{\beta}(s),Y_{\beta}(t)\right).
\end{align*}
The FGPAP exhibits overdispersion as $\operatorname{Var}\left(\bar{\mathcal{M}}_{\beta}(t)\right)-\mathbb{E}\left(\bar{\mathcal{M}}_{\beta}(t)\right)>0$ for all $t>0$.
	
The proof of next result follows similar lines to that of Theorem \ref{thm4.5}.
\begin{theorem}
The FGPAP exhibits the LRD property.
\end{theorem}	
\begin{remark}
As in Remark \ref{srd}, for a fixed $h>0$, the increment process $\bar{\mathcal{Z}}_{\beta}^{h}(t)\coloneqq \bar{\mathcal{M}}_{\beta}(t+h)-\bar{\mathcal{M}}_{\beta}(t)$, $t\ge0$ of FGPAP exhibits the SRD property.	
\end{remark}
\begin{proposition}\label{prop1}
The one-dimensional distributions of FGPAP are not infinitely divisible.
\end{proposition}
The proof of Proposition \ref{prop1} follows on using \eqref{limit}, the self-similarity property of $\{Y_{\beta}(t)\}_{t\ge0}$ in \eqref{ybt31} and the arguments used in Proposition 3 of Kataria and Khandakar (2022).

\end{document}